\documentclass[12pt]{article}
\usepackage{url,mathtools,bm}
\usepackage{enumerate,amsfonts, graphicx, amsmath}
\usepackage{graphicx}
\usepackage{epsfig}
\usepackage{epstopdf}
\usepackage{multirow,color}
\usepackage{algorithm}
\usepackage{algpseudocode}

\newcommand\argmin{\mathop{\rm argmin}}

\newcommand\Span{\mathop{\rm span}}
\newcommand\Diag{\mathop{\rm Diag}}

\algnewcommand{\IIF}[1]{\State\algorithmicif\ #1\ \algorithmicthen}
\algnewcommand{\ENDIIF}{\unskip\ \algorithmicend\ \algorithmicif}

\title{Nonlinear conjugate gradient for smooth
  convex functions\thanks{Supported
in part by a grant  from the Natural Sciences and Engineering Research Council
(NSERC) of Canada.}}
\author{Sahar Karimi\thanks{Department of Combinatorics \& Optimization,
University of Waterloo, 200 University Ave.~W., Waterloo, ON, N2L 3G1,
Canada, {\tt sahar.karimi@gmail.com}.} \and 
Stephen Vavasis\thanks{Department of Combinatorics \& Optimization,
University of Waterloo, 200 University Ave.~W., Waterloo, ON, N2L 3G1,
Canada, {\tt vavasis@uwaterloo.ca.}}}

\renewcommand\b{\bm{b}}

\newcommand\e{\bm{e}}
\newcommand\g{\bm{g}}
\newcommand\p{\bm{p}}

\newcommand\R{\mathbb{R}}
\newcommand\s{\bm{s}}

\renewcommand\u{\bm{u}}

\renewcommand\v{\bm{v}}
\newcommand\w{\bm{w}}
\newcommand\x{{\bm{x}}}
\newcommand\y{\bm{y}}
\newcommand\z{\bm{z}}
\newcommand\eps{\epsilon}

\newcommand\bz{\bm{0}}

\newtheorem{theorem}{Theorem}

\newenvironment{proof}{\noindent {\bf Proof. }}{\hfill $\square$}
\begin{document}
\maketitle
\begin{abstract}
  The method of nonlinear conjugate gradients (NCG)
  is widely used in practice for unconstrained
  optimization, but it satisfies weak complexity bounds at best when
  applied to smooth convex functions.  In contrast, Nesterov's
  accelerated gradient (AG) method is optimal up to constant
  factors for this class.

  However, when specialized to quadratic function, conjugate gradient is
  optimal in a strong sense among function-gradient methods.  Therefore, there
  is seemingly a gap in the menu of available algorithms: NCG,
  the optimal
  algorithm for quadratic functions that also exhibits good practical performance
  for general functions, has poor complexity bounds compared to AG.

  We propose an NCG method
  called C+AG (``conjugate plus accelerated gradient'')
  to close this gap, that is, it is optimal for quadratic functions
  and still satisfies the best possible complexity bound for more
  general smooth convex functions.  It takes
  conjugate gradient steps until insufficient progress is
  made, at which time it switches to accelerated gradient
  steps, and later retries conjugate gradient.  The proposed method has
  the following theoretical properties: (i) It is identical to
  linear conjugate gradient (and hence terminates finitely) if the
  objective function is quadratic;
  (ii) Its running-time bound is $O(\eps^{-1/2})$
  gradient evaluations
  for an $L$-smooth convex function, where $\eps$ is the desired
  residual reduction,
  (iii) Its running-time bound
  is $O(\sqrt{L/\ell}\ln(1/\eps))$
  if the function is both $L$-smooth and $\ell$-strongly convex.
  We also conjecture and outline a proof that a variant of the method has the property:
  (iv) It is $n$-step quadratically convergent for a function
  whose second derivative is smooth and invertible at the optimizer.
  Note that the bounds in (ii) and (iii) match AG and are the
  best possible, i.e., they match lower bounds up to constant factors for
  the classes of functions under consideration.  On the other hand, (i) and
  (iv) match NCG.

  In computational tests, the function-gradient evaluation
  count for the C+AG method typically
  behaves as whichever is better of AG or classical NCG.
  In some test cases it outperforms both.
\end{abstract}


\section{First-order methods for smooth convex functions}

The problem under consideration is minimizing an unconstrained
$L$-smooth convex function
$f:\R^n\rightarrow\R$.  Recall that a convex function is
{\em $L$-smooth}
if it is differentiable and
for all $\x,\y\in\R^n$,
\begin{equation}
f(\y)-f(\x)-\nabla f(\x)^T(\y-\x)\le L\Vert\x-\y\Vert^2/2.
\label{eq:smoothconvdef}
\end{equation}
Note that convexity implies that the quantity on the left-hand side of this
inequality is nonnegative.
We will also sometimes consider {\em $\ell$-strongly convex} functions, which
are defined to be those functions satisfying
\begin{equation}
f(\y)-f(\x)-\nabla f(\x)^T(\y-\x)\ge \ell\Vert\x-\y\Vert^2/2
\label{eq:strongconvdef}
\end{equation}
for all $\x,\y\in\R^n$ for some modulus $\ell\ge 0$.  (This definition
presumes differentiability.  The class of strongly convex
functions includes nondifferentiable functions as well, but those functions
are not smooth and are therefore not relevant to this paper.)
Note that $\ell=0$ 
means simply that $f$ is differentiable and convex.

Nonlinear conjugate gradient (NCG) \cite{NocedalWright} is perhaps the
most widely used first-order method for unconstrained optimization.
It is an extension of linear conjugate gradient (LCG)
introduced by Hestenes and
Stiefel \cite{hestenesstiefel} to minimize convex quadratic functions.
The smoothness modulus $L$ for a quadratic function $f(\x)=\x^TA\x/2 -\b^T\x$,
where $A\in\R^{n\times n}$ is positive semidefinite, is the maximum eigenvalue
of $A$.
LCG is optimal for quadratic functions in a strong sense that
$f(\x_k)$, where $k$ is the $k$th iterate, is minimized
over all choices of $\x_k$ lying in the $k$th Krylov space leading to the
estimate of
$k=O(\ln(1/\eps)\sqrt{L/\ell})$ iterations to reduce the
residual of $f$ by a factor $\eps$.

Stated precisely,
Daniel \cite[Theorem 1.2.2]{Daniel}
proved the following bound for the $k$th iteration of LCG:
\begin{equation}
  f(\x_k)-f^* \le 4 \left(\frac{1-\sqrt{\ell/L}}{1+\sqrt{\ell/L}}\right)^{2k}
  (f(\x_0)-f^*).
  \label{eq:daniel}
\end{equation}
Here and in the remainder of the paper, $f^*$ stands for
$\min_{\x\in\R^n}f(\x)$ and is assumed to be attained.
Therefore, if we define the residual reduction
$$\epsilon_1 := \frac{f(\x_k)-f^*}{f(\x_0)-f^*},$$
then a rearrangement and application of simple bounds in \eqref{eq:daniel}
shows that $\eps_1$ reduction is achieved after $k=\ln(4/\epsilon_1)\cdot \sqrt{L/\ell}/2$
iterations.

However, if the class of
problems is enlarged to include all smooth, strongly convex functions,
then NCG fails to satisfy this complexity bound and could
perform even worse than steepest descent as shown by Nemirovsky and
Yudin \cite{NemYud83}.
Nonetheless, in practice, nonlinear conjugate gradient is difficult
to beat even by algorithms with better
theoretical guarantees, as observed in our experiments
and by others, e.g., Carmon et al.~\cite{Carmon}.

Several first-order methods have been proposed for smooth
strongly convex functions that achieve the running time
$O(\ln(1/\eps)\sqrt{L/\ell})$ 
including
Nemirovsky and Yudin's
\cite{NemYud83} method, which requires solution of
a two-dimensional subproblem on
each iteration, Nesterov's accelerated gradient (AG) \cite{Nesterov:k2}, and
Geometric Descent (GD) \cite{bubeck}.
In the case of GD, the bound
is on iterations rather than gradient evaluations since each iteration
requires additional gradient evaluations for the line search.
It is known
that $\Omega(\ln(1/\eps)\sqrt{L/\ell})$ gradient evaluations are required to
reduce the objective function residual by a factor of $\eps$;
see, e.g., the book of Nesterov \cite{Nesterov:book} for
a discussion of the lower bound result.
Since the algorithms mentioned in this paragraph match the lower bound,
these algorithms are optimal up to a constant factor.
In the case of smooth convex (not strongly convex) functions, AG
requires $O(\eps^{-1/2})$ iterations.

The precise bound for accelerated gradient for
strongly convex functions is stated below in
\eqref{eq:nestrhs}, and this bound can be rearranged
and estimated as follows.
Define the residual reduction
$$\epsilon_2 := \frac{f(\x^k)-f^*}{L\Vert\x_0-\x^*\Vert^2}.$$
This reduction can be attained after no more than
$k=\ln(1/\epsilon_2)\cdot \sqrt{L/\ell}$
iterations.

A technical issue in comparing the bounds is the
different definitions of $\epsilon_1$ and $\epsilon_2$.
A standard result of extremal eigenvalues states that
$$\ell\Vert \x_0-\x^*\Vert^2/2 \le  f(\x_0)-f^* \le
L\Vert \x_0-\x^*\Vert^2/2.$$
Therefore, it follows that $\epsilon_2 \le 2\epsilon_1 \le (L/\ell)\epsilon_2.$
Thus, $\epsilon_1$ tightly controls $\epsilon_2$,
whereas $\epsilon_2$ only loosely controls $\epsilon_1$,
and 
therefore, the LCG bound is slightly better than
the corresponding AG bound in this technical sense.

In practice, conjugate gradient for quadratic functions is typically
superior to AG applied to quadratic functions.  The reason appears to
be not so much the technical reason mentioned in the previous paragraph but
the fact that LCG is optimal in a stronger sense than AG: the iterate $\x_k$ is
the true minimizer of $f(\x)$ over the $k$th Krylov space, which is
defined to be $\{\b,A\b,\ldots,A^{k-1}\b\}$
assuming $\x_0=\bz$.
For quadratic problems in which the eigenvalues lie in a
small number of clusters, one expects convergence much faster than
$O(\ln(1/\eps)\sqrt{L/\ell})$.
But even in the case of spread-out eigenvalues, conjugate gradient often
bests accelerated gradient.
Table~\ref{table:quad} shows the results for these
algorithms for three $1000\times 1000$ matrices:
\begin{align*}
      A_1&=\Diag(\underbrace{1,\ldots,1}_{500},
      \underbrace{1000,\ldots,1000}_{500}),\\
      A_2&=\Diag(\underbrace{1,\ldots,1}_{250},
      \underbrace{500,\ldots,500}_{250},
      \underbrace{1000,\ldots,1000}_{500}),\\
      A_3&=\Diag(1, 4, 9, \ldots, 1000^2),
\end{align*}
two of which have clustered eigenvalues, and the third
has spread-out eigenvalues.

\begin{table}
  \begin{center}
    \caption{Minimization of $f(\x)=\x^TA\x/2-\b^T\x$ for diagonal
      matrices $A_1,A_2,A_3$ described in the text with
     $\b^T=[\sin 1,\sin 2,\ldots, \sin 1000]$.
   Tolerance is $\Vert \nabla f(\x)\Vert \le 10^{-8}$}.
  \label{table:quad}
  \begin{tabular}{lrrrr}
    \hline 
    &  LCG & C+AG & NCG & AG \\
    \hline 
     $A_1$: \# iterations                    & 2   & 3 & 2 & 9167 \\
     $A_1$: \# function-gradient evaluations & 2  & 27 & 5 & 18357 \\
    \hline 
     $A_2$: \# iterations                    &  3  & 4 & 3 & 10267 \\
     $A_2$: \# function-gradient evaluations &  3 & 30 & 7 & 20557 \\
    \hline 
     $A_3$: \# iterations                    &  1509  & 1512 & 1515 & $>10^6$ \\
     $A_3$: \# function-gradient evaluations &  1509  & 3065 & 3036 & $>10^6$ \\
    \hline
  \end{tabular}
  \end{center}
\end{table}

Note that the three conjugate gradient methods require roughly
the same number of iterations, all much less than AG.  The function
evaluations are: 1 per LCG iteration (i.e., a single matrix-vector
multiplication), 2 per C+AG iteration, approximately 2 per NCG iteration,
and 1 per AG iteration.  The requirement for more than one function evaluation
per NCG iteration stems from the need to select a step-size.  We return
to this point below.  Table \ref{table:quad} shows results for the
adaptive versions of C+AG and AG for
the case that $L$ is not known a priori, discussed
in Section~\ref{sec:estL}.  Thus, further function-gradient evaluations 
are needed to estimate $L$, and the number of
function-gradient evaluations is increased accordingly.
The reader will note from Table~\ref{table:quad} that
C+AG required only $2$ iterations for $A_2$ but $27$ function
evaluations due to the overhead of initially estimating $L$, and a similar
pattern is noted for $A_3$.  The adaptive estimation doubles the function-gradient
evaluations for AG from 1 to 2 per iteration.

In the previous paragraph and for the remainder of this paper, we speak of
a ``function-gradient evaluation'', which means, for a given $\x$, the
simultaneous evaluation of $f(\x)$ and $\nabla f(\x)$.  We regard
function-gradient evaluation as the atom of work.  This metric
may be inaccurate in the setting that $f(\x)$ can be evaluated
much more efficiently than $\nabla f(\x)$ for an algorithm that separately uses
$f$ and $\nabla f$.  However, 
in most applications evaluating $f,\nabla f$ together is not much more
expensive than evaluating either separately.  Theoretically, gradient
evaluation is never more than a constant factor larger than function
evaluation due to the existence of reverse-mode automatic differentiation;
see, e.g., \cite{NocedalWright}.

It is also known that NCG is fast for functions
that are nearly quadratic.  In particular, a classic result of
Cohen \cite{Cohen} shows that NCG exhibits $n$-step quadratic
convergence near the minimizer provided that the objective function
is nearly quadratic and strongly convex in this neighborhood.

Thus, there is seemingly a gap in the menu of available algorithms
because NCG, though optimal for quadratic functions and fast
for nearly quadratic functions, is not able to achieve reasonable
complexity bounds for the larger class of smooth, strongly convex
functions.  On the other hand, the methods with the good complexity
bounds for this class are all suboptimal for quadratic functions.

We propose C+AG to close this gap.  The method is based on the following ideas:
\begin{enumerate}
\item
  A line-search for NCG described
  in Section~\ref{sec:linesearch} is used that is exact for quadratic functions but
  could be poor for general functions.  Failure of the line-search is
  caught by the progress measure described in the next item.
\item
  A progress measure described in
  Section~\ref{sec:progress} 
  for nonlinear conjugate gradient
  based on Nesterov's estimate sequence is checked.
  The progress measure is inexpensive to evaluate on every iteration.
\item
  Restarting in the traditional sense of NCG (i.e., an occasional steepest
  descent step) is used, as described
  in Section~\ref{sec:restart}
\item
  On iterations when insufficient progress is made, C+AG switches to
  AG to guarantee reduction in the progress measure as explained
  in Section~\ref{sec:switching}.
\item
  C+AG switches from AG back to NCG after a test measuring nearness to a quadratic
  is satisfied.
  When CG resumes, the progress measure may be modified 
  to a more
  elaborate procedure
  discussed in Appendix~\ref{sec:restartcg}, although this elaborate
  progress measure is disabled by default.
\end{enumerate}
When applied to a quadratic function, the progress measure is
always satisfied (this is proved in Section~\ref{sec:progress}), and hence
all steps will be NCG (and therefore LCG) steps.  The fact that the C+AG satisfies the
$O(\ln(1/\eps)\sqrt{L/\ell})$ running-time bound of AG for general
convex functions follows because
of the progress measure.  The conjectured final theoretical property of the method,
namely, $n$-step quadratic convergence, is described in the
Appendix~\ref{sec:nstep}.
Computational testing in Section~\ref{sec:comp} indicates that the running
time of the method measured in function-gradient evaluations is roughly equal to
whichever of AG or NCG is better for the problem at hand,  and in some cases
it outperforms both of them.

Both AG and C+AG use prior knowledge of $L,\ell$.  However, $L$ can
be estimated by both methods, and $\ell$ may be taken
to be 0 if no further information is available. Thus, both algorithms
are useable, albeit slower,
in the absence of prior knowledge of
these parameters.  We return to this topic
in Section~\ref{sec:estL}.

We conclude this introductory section with a few remarks about our
previous unpublished manuscript
\cite{KarimiVavasis-preprint}.   That work explored connections between
AG, GD, and LCG, and proposed a hybrid of NCG and GD.  However, that work
did not address the number of function-gradient evaluations because the 
the evaluation count for the line search was not analyzed.
Closer to what is proposed here is a hybrid algorithm from the
first author's PhD thesis \cite{Karimi:Thesis}.
Compared to the PhD thesis, the proposed algorithm
herein has several improvements including the line search and the method for
switching from AG to CG as well as additional analysis.

\section{NCG and line search}
\label{sec:linesearch}

The full C+AG algorithm is presented in Section~\ref{sec:pseudocode}.
In this section, we present
the NCG portion
of C+AG, which is identical to most other NCG routines and is as follows
\begin{tabbing}
  +++\=+++\=+++\=\kill
  \> $\x_0:=\bz;$ \\
  \> $\p_0:=-\nabla f(\x_0);$ \\
  \> $k:=0;$ \\
  \> while $\Vert \nabla f(\x_k)\Vert > \mathtt{tol}$ \\
  \> \> Select $\alpha_{k}$; \\
  \> \> $\x_{k+1} := \x_k + \alpha_{k}\p_{k}  ;$ \\
  \> \> Confirm that $\x_{k+1}$ satisfies the progress measure. \\
  \> \> Select $\beta_{k+1}$; \\
  \> \> $\p_{k+1} := -\nabla f(\x_{k+1})+\beta_{k+1}\p_{k};$ \\
  \> \> $k:= k+1$; \\
  \> end while
\end{tabbing}

For $\beta_{k+1}$ we use the Hager-Zhang formula
\cite{HagerZhang}, which is as follows:
\begin{align*}
  \hat\y &:= \g_{k+1} - \g_k; \\
  \beta^1 & :=
  \left(\hat\y - \p_{k+1}\cdot
  \frac{2 \Vert\hat\y\Vert^2}
       {\hat\y^T\p_{k+1}}\right)^T\frac{\g_{k+1}}
       {\hat\y^T\p_{k+1}}, \\
       \beta^2 & := \frac{-1}
            {\Vert\p_{k+1}\Vert\cdot\min\left(.01\Vert\g_0\Vert,
              \Vert\g_{k+1}\Vert\right)}. \\
            \beta_{k+1} &:= \max(\beta^1,\beta^2).
\end{align*}

For $\alpha_{k}$, we use the following formulas.
\begin{align}
  \tilde{\x} & := \x_k + \p_{k}/L, \label{eq:alphak1} \\
  \s &:= L(\nabla f(\tilde{\x}) - \nabla f(\x_k)), \label{eq:alphak2}\\
  \alpha_{k} &:= \frac{-\nabla f(\x_k)^T\p_{k}}{\p_{k}^T\s}. \label{eq:alphak3}
\end{align}
The rationale for this choice of $\alpha_{k}$ is as follows.  In the case
that $f(\x)=\x^TA\x/2-\b^T\x$ (quadratic), $\nabla f(\x_k)=A\x_k-\b$
and $\nabla f(\tilde \x)=A\tilde \x-\b=A\x_k-\b+A\p_{k}/L$, and
therefore $\s=A\p_{k}$ and $\p_{k}^T\s=\p_{k}^TA\p_{k}$.
In this case, $\alpha_{k}$
is the exact minimizer of the univariate quadratic function
$\alpha\mapsto f(\x_k+\alpha\p_{k})$.

In the case of a nonquadratic function, this choice of $\alpha_{k}$
could be inaccurate, which could lead to a poor choice for
$\x_{k+1}$.  The progress measure in the next section
will catch this failure and select a new
search direction $\p_{k}$ in this case.

As for the choice of $\beta_k$, we mention that Dai and Yuan \cite{DaiYuan2001}
carried out an extensive convergence study of different formulas for
$\beta_k$ (not including the Hager-Zhang formula, but covering many
others), and in the end, proposed a certain
hybrid formula for $\beta_k$.
Note that Dai and Kou \cite{DaiKou} propose another family with
the interesting property \cite[Theorem 4.2]{DaiKou} that convergence
is guaranteed for smooth, strongly convex functions, although
no rate is derived.

Our
experiments (not reported here) suggest that the choice of
$\beta_k$ is not critical to the performance of C+AG. 
The Dai-Yuan paper points out that certain choices of $\beta_k$ will
lead to nonconvergence if the line-search fails to satisfy the strong Wolfe conditions.
However, our line-search does not satisfy even the weak Wolfe conditions.
In other words, the criteria used in the previous literature to
distinguish among the choices of formulas for
$\beta_k$ may not be applicable to C+AG.

\section{Progress measure}
\label{sec:progress}

The progress measure is taken from Nesterov's \cite{Nesterov:book} analysis
of the AG method.  The basis of
the progress measure is a sequence of strongly
convex quadratic functions
$\phi_0(\x), \phi_1(\x),\ldots$ called an ``estimate sequence''
in \cite{Nesterov:book}.

The sequence is initialized with
\begin{equation}
\phi_0(\x):=f(\x_0)+\frac{L}{2}
\Vert\x-\v_0\Vert^2,
\label{eq:phi0}
\end{equation}
where $\v_0:=\x_0$.  
Then $\phi_k$ for $k\ge 1$ is defined in terms of its predecessor
via:
\begin{equation}
\phi_{k+1}(\x)=(1-\theta_k)\phi_k(\x) +
\theta_k\left[f(\bar\x_k)+\nabla f(\bar\x_k)^T(\x-\bar\x_{k})+
  \frac{\ell}{2}\Vert\x-\bar\x_k\Vert^2\right].
\label{eq:phiupd}
\end{equation}
Here, $\theta_k\in(0,1)$ and $\bar\x_k\in\R^n$ are detailed
below.
Note that the Hessian of $\phi_0$ and of
the square-bracketed quadratic
function in \eqref{eq:phiupd} are both multiples of $I$, and therefore
inductively the Hessian of $\phi_k$ is a multiple of $I$ for all $k$.
In other words, for
each $k$ there exists $\gamma_{k},\v_{k},
\phi^*_{k}$ such that
$\phi_k(\x)=\phi_k^*+(\gamma_k/2)\Vert\x-\v_k\Vert^2$, where $\gamma_k>0$.
The formulas for $\gamma_{k+1}$, $\v_{k+1}$, $\phi_{k+1}^*$  are straightforward
to obtain from \eqref{eq:phiupd}
given $\gamma_{k}$, $\v_{k}$, $\phi_{k}^*$  as well as
$\theta_k$ and $\bar\x_k$.
These formulas appear in \cite[p.~73]{Nesterov:book} and are repeated here
for the sake of completeness:
\begin{align}
  \gamma_{k+1}&:=(1-\theta_k)\gamma_k+\theta_k\ell, \label{eq:gammak}\\
  \v_{k+1}&:=\frac{1}{\gamma_{k+1}}
    [(1-\theta_k)\gamma_k\v_k+\theta_k\ell\bar\x_k-\theta_k\nabla f(\bar\x_k)],
    \label{eq:v_k} \\
    \phi_{k+1}^* &:= (1-\theta_k)\phi_k^* + \theta_kf(\bar\x_k)
    -\frac{\theta_k^2}{2\gamma_{k+1}}\Vert\nabla f(\bar\x_k)\Vert^2 \notag\\
    &\hphantom{:=}\quad\mbox{}
    +\frac{\theta_k(1-\theta_k)\gamma_k}{\gamma_{k+1}}
    (\ell \Vert\bar\x_k-\v_k\Vert^2/2 + \nabla f(\bar\x_k)^T(\v_k-\bar\x_k)).
    \label{eq:phi_k}
\end{align}
We call the sequence $\bar\x_0,\bar\x_1,\ldots$ the {\em gradient sequence}
because the gradient is evaluated at these points.

To complete the formulation of $\phi_k(\cdot)$, we now specify
$\theta_k$ and $\bar\x_k$.
Scalar $\theta_k$ is selected as the positive root
of the quadratic equation
\begin{equation}
  L\theta_k^2+(\gamma_k-\ell)\theta_k-\gamma_k=0.
  \label{eq:thetaeq}
\end{equation}
The fact that $\theta_k\in (0,1)$ is easily seen by observing
the sign-change in this quadratic over $[0,1]$.
A consequence of \eqref{eq:thetaeq} and \eqref{eq:gammak}
is the identity
\begin{equation}
  \frac{\theta_k^2}{2\gamma_{k+1}} = \frac{1}{2L}.
  \label{eq:thetaLrel}
\end{equation}

Finally, C+AG has three cases for selecting gradient sequence $\bar\x_k$.
The first case
is $\bar\x_k:=\x_k$.  The second and third cases are presented later.

The principal theorem assuring progress is Theorem 2.2.2
of Nesterov \cite{Nesterov:book}, which we restate
here for the sake of completeness.  The validity
of this theorem does not depend on how the gradient sequence $\bar\x_k$ is
selected in the definition of $\phi_k(\cdot)$.

\begin{theorem} (Nesterov)
  Suppose that for each $k=0,1,2,3,\ldots$, there exists
  a point $\x_k$, called the {\em witness},  satisfying
  $f(\x_k)\le \min_\x\phi_k(\x)\equiv \phi_k^*$.  Then
  \begin{equation}
  f(\x_k)-f^*\le L
  \min\left(\left(1-\sqrt{\ell/L}\right)^k, \frac{4}{(k+2)^2}\right)\Vert \x_0-\x^*\Vert^2.
  \label{eq:nestrhs}
  \end{equation}
  \label{thm:nestthm}
\end{theorem}

The sufficient-progress test used for CG is the condition of
the theorem, namely, $f(\x_k)\le \phi_k^*$.
That is, in the case of CG, the witness sequence is also the
main sequence computed by CG.  As long as this test
holds, the algorithm continues to take CG steps, and the right-hand
side of \eqref{eq:nestrhs} assures us that the optimal complexity
bound is achieved (up to constants) in the case of both
smooth convex functions and smooth, strongly convex functions.

The main result of this section is the following theorem, which states
that if the objective function is quadratic, then the progress
measure is always satisfied by conjugate gradient.

\begin{theorem}
  Assume $\v_0=\x_0$.
  If $f(\x)=\frac{1}{2}\x^TA\x-\b^T\x$, where $A$ is positive definite,
  and if the gradient sequence $\bar\x_k$ and witness sequence are
  the same as the CG sequence
  $\bar\x_k:=\x_k$, then $f(\x_k)\le \phi_k^*$ for every iteration
  $k$.
  \label{thm:quad_suffprog}
\end{theorem}

\begin{proof}
  The result clearly holds when $k=0$ by \eqref{eq:phi0}.  
  Assuming $f(\x_k)\le \phi_k^*$, we now show 
  that $f(\x_{k+1})\le \phi_{k+1}^*$
  For $k\ge 0$,
  \begin{align}
    \phi_{k+1}^*
    &=
    (1-\theta_k)\phi_k^*+\theta_k f(\x_k)
    -\frac{\theta_k^2}{2\gamma_{k+1}}\Vert \nabla f(\x_k)\Vert^2 \notag\\
    &\hphantom{}\quad\mbox{}
    +\frac{\theta_k(1-\theta_k)\gamma_k}{\gamma_{k+1}}
    \left(\ell \Vert \x_k-\v_k\Vert^2+
    \nabla f(\x_k)^T(\v_k-\x_k)\right)\label{eq:cgok1} \\
    &\ge
    f(\x_k)
    -\frac{\theta_k^2}{2\gamma_{k+1}}\Vert \nabla f(\x_k)\Vert^2
    +\frac{\theta_k(1-\theta_k)\gamma_k}{\gamma_{k+1}}
    \left(\nabla f(\x_k)^T(\v_k-\x_k)\right) \label{eq:cgok2}\\
    & = 
    f(\x_k)
    -\frac{1}{2L}\Vert \nabla f(\x_k)\Vert^2
    +\frac{\theta_k(1-\theta_k)\gamma_k}{\gamma_{k+1}}
    \left(\nabla f(\x_k)^T(\v_k-\x_k)\right).\label{eq:cgok3} \\
    & =
    f(\x_k)
    -\frac{1}{2L}\Vert \nabla f(\x_k)\Vert^2\label{eq:cgok4} \\
    & \ge 
    f(\x_k - \nabla f(\x_k)/L)
    \label{eq:cgok5} \\
    & \ge
    f(\x_{k+1}).\label{eq:cgok6}
  \end{align}
  Here, \eqref{eq:cgok1} is a restatement of \eqref{eq:phi_k}
  under the assumption that
  $\bar\x_k=\x_k$.
  Line \eqref{eq:cgok2} follows
  from the induction hypothesis.  Line \eqref{eq:cgok3}
  follows from \eqref{eq:thetaLrel}.

  Line \eqref{eq:cgok4} follows because,
  according to the recursive formula
  \eqref{eq:v_k},
  $\v_k$ lies in the affine space
  $\x_0+\Span\{\nabla f(\x_0), \cdots,\nabla f(\x_{k-1})\}$.
  Similarly, $\x_k$ lies in this space.
  Therefore, $\v_k-\x_k$ lies the $k$th Krylov subspace
  $\Span\{\nabla f(\x_0), \cdots,\nabla f(\x_{k-1})\}$.
  It is well known (see, e.g., \cite{GVL}) that
  $\nabla f(\x_k)$ is orthogonal to every vector in this space.

  Line \eqref{eq:cgok5} is a
  standard inequality for $L$-smooth
  convex functions; 
  see,  e.g., p.~57 of \cite{Nesterov:book}.
  Finally \eqref{eq:cgok6}
  follows
  because $\x_k-\nabla f(\x_k)/L$ lies in the $k+1$-dimensional affine
  space $\x_0+\Span\{\nabla f(\x_0),\ldots, \nabla f(\x_k)\}$
  while $\x_{k+1}$ minimizes $f$ over the same affine space,
  another well known property of CG.
\end{proof}

We note that this proof yields another proof of Daniel's
result that conjugate gradient converges at a linear rate with a factor
$1-\mbox{const}\sqrt{\ell/L}$ per iteration.  Specifically, the two
bounds are
$$f(\x_k)-f(\x^*) \le
\left\{
\begin{array}{l}
  4 \left(\frac{1-\sqrt{\ell/L}}{1+\sqrt{\ell/L}}\right)^{2k}
  (f(\x_0)-f^*) \mbox{ (Daniel)} \\
  \left(1-\sqrt{\ell/L}\right)^k \cdot L\Vert\x_0-\x^*\Vert^2  
  \mbox{ (Nesterov + this proof)}.
\end{array}
\right.$$
We see that the reduction factor
is better in Daniel's bound because of the exponent
$2k$ rather than $k$, which means approximately reduction of
$1-2\sqrt{\ell/L}$ per iteration (assuming $\ell\ll L$)
instead of $1-\sqrt{\ell/L}$.
Furthermore, the multiplicative constant in Daniel's bound is better
as noted earlier.

A limitation of this proof is that it requires $\v_0=\x_0$.  In the
case that the CG iterations follow a sequence of AG iterations, this
equality will not hold.  We return to this point in
Appendix~\ref{sec:restartcg}.

\section{Restarting CG}
\label{sec:restart}
The classical method for coping with failure to make progress
in nonlinear CG is `restarting', which means taking a step using
the steepest descent direction instead of the conjugate
gradient direction.  Our proposed C+AG method, when it detects insufficient
progress on the $k$th iteration 
(i.e., the inequality $f(\x_k)\le \phi_k^*$
fails to hold), attempts two fall-back solutions.  The first
fall-back is restarting CG in the traditional manner (steepest descent).  It should
be noted that this restart does not play a role in the theoretical
analysis of the algorithm, but we found that in practice it
improves performance
(versus immediately proceeding to the second fall-back of
the AG method).

The issue of when to restart has received significant attention
in the previous literature.  One guideline is
to restart every $n+1$ iterations.  In more modern codes,
restarting is done every
$kn+1$ iterations, where $k\ge 1$ is a constant, e.g., Hager-Zhang's
CG-Descent restarts every $6n+1$ iterations.  We also have implemented this
restart in our code, although it is usually not activated because
some other test fails first.

Three other conditions for restarting from the previous literature,
collectively termed the ``Beale-Powell'' conditions by
Dai and Yuan \cite{DY98}, are
\begin{equation}
  |\nabla f(\x_{k-1})^T\nabla f(\x_k)| \ge c_1 \Vert \nabla f(\x_k)\Vert^2,
  \quad \mbox{(BP1)}
  \label{eq:bp1}
\end{equation}
\begin{equation}
  \p_k^T\nabla f(\x_k) \ge -c_2 \Vert\nabla f(\x_k)\Vert^2,
  \quad \mbox{(BP2)}
  \label{eq:bp2}
\end{equation}
\begin{equation}
  \p_k^T\nabla f(\x_k) \le -c_3 \Vert\nabla f(\x_k)\Vert^2,
  \quad \mbox{(BP3)}
  \label{eq:bp3}
\end{equation}
where $c_1\in(0,\infty)$, $c_2\in[0,1]$, and $c_3\in (1,\infty)$.  It should be
noted that in exact arithmetic on a quadratic objective function, none of these
tests will ever be activated.  It should also be noted that
Powell \cite{powell77} as well as
Dai and Yuan consider other restart directions
(which we have not implemented)
besides the
steepest descent direction.

We have found that including these tests (BP1)--(BP3) caused the code
to take more iterations in the majority of tests (although the
difference in iteration counts
was never large), so we have turned them
off by default.  In other words, the default values are $c_1=\infty$,
$c_2=0$ (i.e., we restart if $\p_k$ is not a descent direction),
and $c_3=\infty$.  Presumably this is because there are cases when progress
is possible according to our progress measure even though a BP condition fails.

One additional restart condition in our code is that if the denominator of
\eqref{eq:alphak3} is nonpositive, then we restart because this
means that calculation of $\alpha_k$ has failed.

A recent paper by Buhler et al.\ \cite{buhler2021} considers nonlinear conjugate gradient
for problems in machine learning and in particular the role of (BP) restarts.
The motivation for their work is somewhat different from ours: they
consider general objective functions, not necessarily smooth convex
functions, and therefore accelerated gradient is not available for their
class of problems.  Their first main concern, costly line search, is not an issue
for C+AG since its line search requires only one additional gradient per
iteration.  Their second main concern,
nonconjugacy in search directions,  is also less of a concern for C+AG since
conjugacy is not an ingredient of the C+AG progress measure.
They conclude that (BP1)--(BP3) (and (BP1) in particular)
are important for a successful nonlinear CG
implementation.  The fact that their conclusions
differ from ours arises from the different aims of the papers.

\section{Switching to AG in the case of lack of progress}
\label{sec:switching}

If the inequality $f(\x_k)\le \phi_k^*$ fails to hold also for the
steepest descent
direction described in the previous section, then the second fall-back is
to switch to a sequence of AG steps.

AG steps continue until the following ``almost-quadratic'' test is satisfied.
Let
  $$q_k := -\frac{\nabla f(\bar\x_k)^T(\nabla f(\bar\x_k) + \nabla f(\x_{k+1}))}{2L}.$$
The termination test for AG is
\begin{equation}
  f(\x_{k+1}) \le f(\bar\x_k) + \frac{4}{5}q_k.
  \label{eq:agterm}
\end{equation}
The rationale for this formula is as follows.  One checks with a few lines of
algebra that if $f$ were quadratic, then the identity $f(\x_{k+1})=f(\bar\x_k)+q_k$
holds, assuming that $\x_{k+1}$ is defined as $\bar\x_k-\nabla f(\bar \x_k)/L$ as in
AG (see l.~\ref{stmt:agmain} of Algorithm C+AG, Part II below ).
Thus, \eqref{eq:agterm} tests whether $f$ in
the neighborhood of $\bar\x_k$ is behaving approximately as a quadratic function
would behave, which indicates that conjugate gradient might succeed.

Note that \eqref{eq:agterm} involves an extra function-gradient evaluation
since an ordinary AG iteration evaluates $f,\nabla f$ at $\bar\x_k$ but
not at $\x_{k+1}$. Therefore, \eqref{eq:agterm} is checked only every 8 iterations
to reduce the overhead of the test.

Overall, the C+AG algorithm requires two function-gradient evaluations for a CG
step, one to obtain the gradient at $\x_k$ and the second to
evaluate $\nabla f(\tilde\x)$ in \eqref{eq:alphak2}.  Similarly, the first fall-back requires
two function-gradient evaluations.  Finally, the accelerated gradient
step requires another gradient evaluation.  Therefore, the maximum
number of function-gradient evaluations on a single iteration
before progress is made is
five.  Therefore, Theorem~\ref{thm:nestthm} applies to our algorithm
except for a constant factor of 5.  In practice, on almost every
step, either the CG step succeeds or AG is used, meaning the
actual number of function-gradient evaluations per iteration
lies between 1 and 2.

\section{Estimating $L$}
\label{sec:estL}

The C+AG method as described thus far requires knowledge of $L$ and $\ell$, the
moduli of smoothness and strong convexity respectively.  In case these parameters
are not known, the method estimates $L$ and simply assumes $\ell=0$.  For many
problems (including our examples) $\ell\ll L$,
meaning that convergence is governed by the second parenthesized of
factor of  \eqref{eq:nestrhs}.
Thus, taking $\ell=0$
does not appear to significantly impede convergence.

As for estimating $L$, we use the following procedure
outlined in \S10.4.2 of Beck \cite{Beck}.  Starting from an initial guess
of $L=1$, we first decrease $L$ by multiplying by $1/\sqrt{2}$ or increase by multiplying by $\sqrt{2}$ until the
condition
$$f(\x_0-\nabla f(\x_0)/L) \le f(\x_0) -  \Vert \nabla f(\x_0)\Vert^2/(2L)$$
is satisfied.  This means that C+AG and AG require
$O(|\log(L)|)$ iterations for the initial estimate.
On subsequent iterations, we re-estimate $L$ on the first iteration of
consecutive sequences of CG iterations and on every iteration of AG.
On iterations after the first, $L$ cannot decrease;
it will either stay the same or increase.
Beck proves that using this procedure does not worsen the iteration-complexity
of AG when compared to the case that $L$ is known a priori.  However, this procedure
increases the number of function-gradient evaluations per AG step from one
to approximately two. In the case of the conjugate gradient steps, the impact
is negligible since only the first in a sequence requires the extra function-gradient
evaluation.

It is not
necessary to re-estimate $L$ on every CG iteration because the CG step
uses $L$ only as the finite-difference step size in its procedure to
determine
$\alpha_{k}$ as in
\eqref{eq:alphak1}--\eqref{eq:alphak3}.  This dependence is mild in the sense
that, assuming $f$ is quadratic,  \eqref{eq:alphak1}--\eqref{eq:alphak3}
will compute the same (correct) value of $\alpha_k$
regardless of the step size.  All 
computational tests in Tables~\ref{table:quad} and \ref{tab:results} use
the adaptive version of C+AG
described in this section.

There are procedures in the literature for estimating $\ell$ also, for example, 
the procedure by Carmon et al.\ \cite{Carmon}.  However, Carmon et al.'s
procedure is costly
as it involves several trial steps and a guess-and-update procedure for $\ell$.
Even in the pure quadratic case, estimating $\ell$ apparently
requires multiple steps of the Lanczos method \cite{GVL} and would be costly.
We have not implemented it since, as mentioned above, taking $\ell=0$ appears
to be satisfactory.  

\section{Pseudocode for C+AG}
\label{sec:pseudocode}
In this section we present pseudocode for the
C+AG algorithm.  The pseudocode includes the procedure for
estimating $L$. It does not, however, include the
elaborate procedure for finding witnesses after CG is restarted
as described in Appendix~\ref{sec:restartcg}.  We did not include
this in the pseudocode because it is disabled by default as discussed
in the appendix.
See the concluding remarks of the paper for a downloadable
implementation in Matlab, which does include that
procedure as an option.  The input arguments are $f$ and $\nabla f$,
the objective function and gradient, $\x_0$, the initial guess,
$\texttt{gtol}$, the termination
criterion, and $L$, $\ell$, the smoothness modulus and the modulus of strong
convexity respectively.  Note that $\ell=0$ is a valid argument,
and $L=\mathrm{NaN}$ (not a number) is a signal to the algorithm to
estimate $L$.

\begin{algorithm}
  \caption{ComputeThetaGamma (implementation of \eqref{eq:gammak} and \eqref{eq:thetaeq})}
  \label{alg:thetagamma}
  \begin{algorithmic}
    \Require $L,\ell,\gamma_k$
    \Ensure $\theta_k, \gamma_{k+1}$
    \State Solve $L\theta_k^2+(\gamma_k-\ell)\theta_k-\gamma_k=0$ via
    the quadratic formula for the positive root $\theta_k$.
    \State  Let $\gamma_{k+1}:=(1-\theta_k)\gamma_k+\theta_k\ell$
  \end{algorithmic}
\end{algorithm}

\begin{algorithm}
  \caption{EstimateLInitial}
  \label{alg:estL0}
  \begin{algorithmic}
    \Require $f(\cdot)$, $\x_0$, $f(\x_0)$,
    $\nabla f(\x_0)$, $L_{\rm init}(=1\mbox{ by default})$
    \Ensure $L$
    \State $L:=L_{\rm init}$
    \For {$k = 1 : 100$}
    \If {$f(\x_0 - \nabla f(\x_0) / L) < f(\x_0) - \Vert \nabla f(\x_0)\Vert^2/(2L)$}
    \State $L := L /\sqrt{2}$
    \Else
    \State \textbf{return} \texttt{EstimateL}$(f(\cdot), \x_0, f(\x_0), \nabla f(\x_0), L)$
    \EndIf
    \EndFor
    \State \textbf{Throw error}: ``$f$ may be unbounded below''
  \end{algorithmic}
\end{algorithm}

\begin{algorithm}
  \caption{EstimateL}
  \label{alg:estL}
  \begin{algorithmic}
    \Require $f(\cdot)$, $\x_0$, $f(\x_0)$, $\nabla f(\x_0)$, $L$
    \Ensure $L$
    \For {$k=1:60$}
    \State $f_1:=f(\x_0-\nabla f(\x_0)/L) $
    \If {$f_1 \ge f(\x_0) - \Vert \nabla f(\x_0)\Vert^2/(2L)$
      \textbf{and} $|f_1-f(\x_0)|\ge 10^{-11}|f(\x_0)|$}
    \State $L := \sqrt{2}L$
    \Else
    \State \textbf{return} $L$
    \EndIf
    \EndFor
    \State \textbf{Throw error}: ``Line search failed to determine $L$; possible
    incorrect gradient function or excessive roundoff error''
  \end{algorithmic}
\end{algorithm}

\begin{algorithm}
  \caption{ComputeVPhiStar (implementation of \eqref{eq:v_k} and \eqref{eq:phi_k})}
  \label{alg:vphistar}
  \begin{algorithmic}
    \Require $\theta_k,\gamma_k,\gamma_{k+1},\ell,\v_k,\phi_k^*,\bar\x_k,
    f(\bar\x_k),\nabla f(\bar\x_k)$
    \Ensure $\v_{k+1},\phi_{k+1}^*$
    \State Let $\v_{k+1}:=\frac{1}{\gamma_{k+1}}
           [(1-\theta_k)\gamma_k\v_k+\theta_k\ell\bar\x_k-\theta_k\nabla f(\bar\x_k)].$
    \State
       $\begin{aligned}           
           \mbox{Let }\phi_{k+1}^* := &(1-\theta_k)\phi_k^* + \theta_kf(\bar\x_k)
           -\frac{\theta_k^2}{2\gamma_{k+1}}\Vert\nabla f(\bar\x_k)\Vert^2 \\
           &\>\mbox{}
    +\frac{\theta_k(1-\theta_k)\gamma_k}{\gamma_{k+1}}
    (\ell \Vert\bar\x_k-\v_k\Vert^2/2 + \nabla f(\bar\x_k)^T(\v_k-\bar\x_k)).
    \end{aligned}
    $
  \end{algorithmic}
\end{algorithm}

\begin{algorithm}
  \caption{C+AG (Part I)}
  \label{alg:C+AG}
  \begin{algorithmic}[1]
    \Require $f(\cdot),\nabla f(\cdot), \x_0\in\R^n, \ell, L, \texttt{gtol}$
    \Ensure $\x^*$ such that $\Vert\nabla f(\x^*)\Vert\le \texttt{gtol}$
    \State $f_0:=f(\x_0)$; $\g_0:=\nabla f(\x_0)$; 
    $\phi_0^*:=f_0$; $\v_0:=\x_0$;  $\texttt{onlyAG}:=\mathbf{false}$;
    \State $i_{cg}:=0$; $i_{ag}:=0$; 
    $\p_0:=-\g_0$; $\texttt{LNaNFlag}:=\mathbf{false}$
    \If {$L=\mathrm{NaN}$}
    \State $L:=\texttt{EstimateLInit}(f(\cdot),
    \x_0, f(\x_0), \nabla f(\x_0), 1)$;
    $\ell:=0$;
    $\texttt{LNaNFlag}:=\mathbf{true}$
    \EndIf
    \State $\gamma_0:=L$;
    \For {$k = 0,1,\ldots$}
    \State $\theta_{k},\gamma_{k+1}:=\texttt{ComputeThetaGamma}(L,\ell,\gamma_k)$
    \For {$\texttt{whichsteptype}:=1:3$} \label{stmt:cgattempt}
    \If {$\texttt{whichsteptype}\le 2$}
    \IIF {$\texttt{onlyAG}$}
    \textbf{continue} \ENDIIF
    \If {$i_{cg}\ge 6n+1$\textbf{ or } $\texttt{whichsteptype}=2$}
    \State $\p_{k}:=-\g_k$; $i_{cg}:=0$ 
    \EndIf
    \If {$i_{cg}=0$ \textbf{and} $k>0$ \textbf{and} \texttt{LNaNFlag}}
    \State  $L:=\texttt{EstimateL}(f(\cdot),
    \x_k, f(\x_k), \nabla f(\x_k), L)$
    \EndIf
    \State $i_{cg}:=i_{cg}+1;\> i_{ag}:=0$
    \State $\tilde\x:=\x_k+\p_{k}/L;$ $\tilde f:=f(\tilde\x)$;
    $\tilde\g :=\nabla f(\tilde\x)$
    \IIF {$\Vert \tilde\g\Vert\le \texttt{gtol}$} $\mathbf{return}\>\tilde\x$ \ENDIIF
    \State $\texttt{Ap} := L(\tilde\g-\g_k)$;
    $\texttt{pAp} := \p_{k}^T\texttt{Ap}$;
    \IIF {$\g_k^T\p_k\ge 0 \textbf{ or } \texttt{pAp}\le 0$} \textbf{continue}\ENDIIF
    \State $\alpha_{k}:=-\g_k^T\p_{k}/\texttt{pAp}$ \label{stmt:alpha}
    \State $\x_{k+1} := \x_k+\alpha_{k}\p_{k}$; $f_{k+1}:=f(\x_{k+1})$;
    $\g_{k+1}:=\nabla f(\x_{k+1})$
    \IIF {$\Vert \g_{k+1}\Vert\le \texttt{gtol}$} $\mathbf{return}\>\x_{k+1}$ \ENDIIF
    \State  $\v_{k+1},\phi_{k+1}^*:=
    \texttt{ComputeVPhiStar}(\theta_k,\gamma_k,\gamma_{k+1},
    \ell,\v_k,\phi_k^*,\x_k, f_k,\g_k)$
    \If {$f_{k+1}\le \phi_{k+1}^*$} \label{stmt:suffprog}
    \State  $\hat\y := \g_{k+1} - \g_k$ 
    \State $\beta^1  :=
    \left(\hat\y - \p_{k}\cdot
    \frac{2 \Vert\hat\y\Vert^2}
         {\hat\y^T\p_{k}}\right)^T\frac{\g_{k+1}}
         {\hat\y^T\p_{k}}$
    \State  $\beta^2 := \frac{-1}
            {\Vert\p_{k}\Vert\cdot\min\left(.01\Vert\g_0\Vert,
              \Vert\g_{k+1}\Vert\right)}.$
    \State $\beta_{k+1} := \max(\beta^1,\beta^2)$
    \State $\p_{k+1}:=-\g_{k+1}+\beta_{k+1}\p_{k}$
    \State \textbf{break} \mbox{(terminate \texttt{whichsteptype} loop)}
    \EndIf
    \algstore{cplusag}
  \end{algorithmic}
\end{algorithm}

\begin{algorithm}
  \caption{C+AG Part II}
  \begin{algorithmic}[1]
    \algrestore{cplusag}
    \Else $\>\>(\texttt{whichsteptype}=3)$
    \If {$\mathbf{not}\>\texttt{onlyAG}$} \label{stmt:ag}
    \State $\texttt{onlyAG}:=\mathbf{true}$
    \State $i_{ag} := 0$; $i_{cg} :=0$
    \EndIf
    \State $i_{ag}:=i_{ag}+1$
    \State $\bar{\x}_k :=
    \frac{\theta_k\gamma_k\v_k+\gamma_{k+1}\x_k}
         {\gamma_k+\theta_k\ell}$;  \label{stmt:ag2}
    $\bar f_k:=f(\bar{\x}_k)$;
    $\bar \g_k:=\nabla f(\bar{\x}_k)$
    \IIF {$\Vert \bar\g_k\Vert\le \texttt{gtol}$}
    \textbf{return} $\bar\x_k$ \ENDIIF
    \IIF {\texttt{LNaNFlag}}
     $L:=\texttt{EstimateL}(f(\cdot),
     \x_k, f(\x_k), \nabla f(\x_k), L)$
     \ENDIIF
    \State $\x_{k+1}:=\bar{\x}_k-\bar{\g}_k/L$ \label{stmt:agmain}
    \State  $\v_{k+1},\phi_{k+1}^*:=
    \texttt{ComputeVPhiStar}(\theta_k,\gamma_k,\gamma_{k+1},
    \ell,\v_k,\phi_k^*,\bar\x_k,
    \bar f_k,\bar\g_k)$
    \If {$i_{ag}\equiv 0\> (\mbox{mod $8$})\> \mathbf{and} \>
  f(\x_{k+1}) \le f(\bar\x_k) - \frac{4}{5}\cdot
  \frac{\bar\g_k^T(\bar\g_k + \nabla f(\x_{k+1}))}{2L}$} \label{stmt:nearquadtest}
    \State  $f_{k+1}:=f(\x_{k+1})$; \label{stmt:cgprep}
    $\g_{k+1}:=\nabla f(\x_{k+1})$
    \State $\p_{k+1}:=-\g_{k+1}$
    \State $\texttt{onlyAG}:=\mathbf{false}$
    \EndIf
    \EndIf
    \EndFor
    \EndFor
  \end{algorithmic}
\end{algorithm}

The mathematical structure
of the algorithm has already been described in previous sections, and in this
section we will describe a few computational details.
The flag \texttt{onlyAG} indicates that the algorithm is currently executing
a block of AG statements.
As mentioned in Section~\ref{sec:switching}, once AG iterations
begin, they continue uninterrupted until
\eqref{eq:agterm} holds, and this is checked every 8th iteration.

The variable $i_{cg}$ counts the number of consecutive CG iterations without
a restart.  A restart means that $\p_{k+1}$ is defined to be $-\g_k$, i.e.,
a steepest descent step is taken.  The code forces a restart after $6n+1$ iterations.
Such a test is typical in NCG algorithms.  The theory for $n$-step quadratic convergence
requires a restart every $n+1$ iterations, but well known codes, e.g.,
CG-Descent \cite{HagerZhang} often use a longer interval.

The for-loop on l.~\ref{stmt:cgattempt} tries first a CG step and second
a steepest descent step in order to make progress.  Line \ref{stmt:alpha}
computes $\alpha_{k}$ as discussed in Section~\ref{sec:linesearch}.
The test for sufficient progress in a CG step appears in l.~\ref{stmt:suffprog}.
If sufficient progress is attained, then the next CG search direction
$\p_{k+1}$ is computed.

Line \ref{stmt:ag} can be reached if both attempts at CG fail to make
sufficient progress.  It is also reached if the \texttt{onlyAG} flag is
set to `true', meaning that the algorithm is currently inside a block of consecutive
AG iterations.  The statements beginning with \ref{stmt:ag2} are the formulation
of the AG method taken from \cite{Nesterov:book}.
Statement l.~\ref{stmt:nearquadtest} implements the test for resuming CG
described in Section~\ref{sec:switching}.

\section{Computational experiments}
\label{sec:comp}

In this section we report on computational experiments with C+AG. We compared
it to two other codes: AG coded by us based on (2.2.8) of \cite{Nesterov:book}
and CG-Descent by Hager and Zhang.  We used C version 6.8 of CG-Descent
with the Matlab front-end.  We used it in memoryless mode, i.e., NCG rather
than L-BFGS.

As for our own code C+AG, we discovered via experimentation that the
more elaborate calculation of $\bar\x_k$ described in Appendix~\ref{sec:restartcg}
diminished the performance of the code.
In particular, as mentioned
the additional overhead of 100\% per iteration
slowed the code down without improving the iteration count and thus was not used.

To the best of our knowledge, there is no standard benchmark set of
smooth convex functions with a known $L,\ell$, so instead we constructed
our own test set of eight problem instances based on formulations
that have occurred in the recent literature in data science.

The convergence criterion is $\Vert \nabla f(\x_k)\Vert\le \texttt{gtol}$,
where $\texttt{gtol}=10^{-8}$ for the ABPDN and LL problems described below,
while $\texttt{gtol}=10^{-6}$  for the HR problems.
(None of the algorithms
were able to converge to $10^{-8}$ for the HR problems in fewer
than $2\cdot 10^6$ iterations, so we loosened the criterion for that problem.)
In the case of CG-Descent, the convergence criterion is based on the
$\infty$-norm rather than 2-norm.  (The $\infty$-norm convergence criterion is hardwired
into the code, and we did not attempt to rewrite this fairly complex
software.)
Via some experiments not reported
here, we found that stopping when
$\Vert \nabla f(\x_k)\Vert_\infty\le 3\cdot \texttt{gtol}/
\sqrt{n}$ approximately reproduced the criterion
$\Vert \nabla f(\x_k)\Vert\le \texttt{gtol}$ used for the other methods, so this
was the termination criterion for CG-Descent.

The reported results for our C+AG routine used $L=\mathbf{NaN}$ and $\ell=0$
in all cases, in other words, adaptive estimation of $L$.  We assume most
users would select this mode since $L,\ell$ are often not known in advance.
For AG, we tested both the known and unknown $L,\ell$ modes; this is indicated in the
table below by ``AG'' and ``AG/EstL.''

As mentioned in the introduction, we have reported on function-gradient evaluation
counts since that is usually taken as the primary work-unit of first-order methods.
We did not report on wall-clock time because CG-Descent, having been written in
highly optimized C, has an advantage over AG and C+AG, which are written Matlab.

The first four problems are
based on BPDN (basis pursuit denoising), that is, the
unconstrained convex optimization
problem:
$$\min \frac{1}{2}\Vert A\x-\b\Vert^2 +\lambda\Vert\x\Vert_1$$
in which $\lambda > 0$ and $A\in\R^{m\times n}$ 
has fewer rows than columns, so that the problem
is neither strongly convex nor smooth.  However, the following
approximation (called APBDN)
is both smooth and strongly convex on any bounded domain:
$$\min \frac{1}{2}\Vert A\x-\b\Vert^2 +\lambda\sum_{i=1}^n\sqrt{x_i^2+\delta}$$
where $\delta>0$ is a fixed scalar.  It is easy to see that as
$\delta\rightarrow 0$, the original problem is recovered.  As
$\delta\rightarrow 0$, $\ell\rightarrow 0$ and $L\rightarrow\infty$,
where $\ell,L$ are the moduli of strong and smooth convexity respectively.

In our tests of ABPDN 
we took $A$ to be a subset of $\sqrt{n}$ rows
of the discrete-cosine transform matrix of size $n\times n$, where
$n$ is an even power of 2.  (This matrix and its transpose, although
dense, can be applied in $O(n\ln n)$ operations.)  The subset of
rows was selected to be those numbered by the first $m=\sqrt{n}$ prime
integers
in order to get reproducible
pseudorandomness in the choices.  Similarly, in 
order to obtain a pseudorandom $\b$, we selected $\b\in\R^m$ 
according to the formula $b_i=\sin(i^2)$.  The value of
$\lambda$ was fixed at $10^{-3}$ in all tests.
Finally,
we varied $\delta=10^{-4}$ and $\delta=5\cdot 10^{-6}$ and we tried both
$n=65536$ and $n=262144$ for a total of four test cases.

The $\ell_1$-penalty term in the original BPDN formulation leads to
many entries of the solution $\x^*$ equal to 0.  This is no longer the
case for ABPDN since the nondifferentiability at $x_i=0$ has
been smoothed out.  However,
a simple scaling argument suggests that the threshold
for ``nearly zero'' should be magnitude bounded above by
$\sqrt{\delta}$.  We determined
that our computed solutions for the four ABPDN solutions had between
78\% to 82\% of entries nearly zero.

The second test case is logistic loss (LL), which is as follows:
$$f(\x)=R(A\x)+\lambda\Vert\x\Vert^2/2,$$
where
$A$ is a given $m\times n$
matrix,
and $\lambda>0$ is a regularization parameter.  Here,
$R(\v)=\sum_{i=1}^mr(v_i)$ where  $r(v)=\ln(1+e^{-v})$.
Row $i$ of $A$, $i=1:n$, is of the form $\e^T/\sqrt{n}+\z_i$,
where $\e\in\R^n$ is the vector of all 1's and
$\z_i\sim N(\bz, \sigma^2I)$ where $\sigma=0.4$.  We seeded
Matlab's random number generator with the same seed on each run
for reproducibility.
This formulation arises from the problem of identifying the
the best halfspace that contains noisy data points.
We ran this test with $A\in\R^{6000\times 3000}$ with
two
values of $\lambda$, namely, $\lambda=10^{-4}$ and
$\lambda = 5\cdot 10^{-6}$.  This function is smooth and strongly
convex.  The function $\lambda$ controls the strong convexity.

The third test case is Huber regression (HR).
Given a matrix $A\in\R^{m\times n}$
and a vector $\b\in\R^m$, Huber regression
\cite{huber} has been proposed as
a means to make linear least-squares regression more robust against
outliers.
In more detail, a cutoff $\tau>0$ is selected in order to define the
function
$$\zeta(t)=\left\{
\begin{array}{ll}
  -\tau^2-2\tau t, & t\le -\tau, \\
  t^2, & t\in[-\tau,\tau], \\
  -\tau^2+2\tau t, & t\ge \tau.
\end{array}
\right.$$
Note that $\zeta(t)$ is differentiable and convex
and behaves like $t^2$ for small $|t|$ and like $O(|t|)$ for large
$|t|$.
Finally, $f(\x)=\sum_{i=1}^m\zeta(A(i,:)\x-b_i).$
This function is smooth but not strongly convex, i.e., $\ell=0$.
We chose the $(n+1)\times n$ sparse matrix with $1$'s on the main
diagonal and $-1$'s on the first sub diagonal.  The vector
$\b$ was taken to be all 1's except for the final entry, which
is $-1.1n$.  (This $\b$ is close to
$A\x$ where $\x=[1,2,\ldots,n]^T$; in other words, $\b$ is close to a right-hand
side that has a 0-residual solution.)
We fixed $n=10000$.
We selected two choices for $\tau$, namely, $\tau=250$ and $\tau=1000$.

Our results for the four algorithms on the eight problems are shown
in Table~\ref{tab:results}.  We set a limit of $10^6$ function-gradient
evaluations for C+AG and AG.

\begin{table}
  \begin{center}
  \caption{Function-gradient evaluation counts for four algorithms
    in columns 3--6 on eight problems.
    Bold indicates the best for each row.  The second
    column of the table indicates the percentage of C+AG iterations spent in
    AG iterations.}
  \label{tab:results}
  \begin{tabular}{|l|r|r|r|r|r|}
    \hline
    Problem & {\%}AG & C+AG & AG & AG/EstL & CG-Descent \\
    \hline
    ABPDN$(n=65536,\delta=10^{-4}) $ & .03\% & \textbf{55,891} & 518,019 & 982,919 &
    82,472\\
    ABPDN$(n=65536,\delta=5\cdot 10^{-6})$ & 14\% &
    226,141 & 660,355 & $>10^6$  & \textbf{165,207} \\
    ABPDN$(n=262144, \delta = 10^{-4})$  & .02\% &
    \textbf{80,335} & 901,418 &
    $>10^6$ & 130,040 \\
    ABPDN$(n=262144, \delta = 5\cdot 10^{-6}) $  & 0\%  &
    \textbf{483,420} & $>10^6$ & $>10^6$ & 532,706  \\
    LL$(\lambda=10^{-4})$ &  0\% &
    148 & 106,507 & $>10^6$
    & \textbf{128} \\
    LL$(\lambda=5\cdot 10^{-6})$ & 0\%&
    140 & 362,236 & $>10^6$  & \textbf{125} \\
    HR$(\tau=250)$ &   64\% &
    \textbf{160,115} &  $> 10^6$ & $>10^6$ & 946,488 \\
    HR$(\tau=1000)$ & 60\% & 
    \textbf{95,416} & $>10^6$ & $>10^6$ & 245,376 \\
    \hline
  \end{tabular}
  \end{center}
\end{table}

The results in Table~\ref{tab:results} show that
C+AG outperformed both AG and CG-Descent in most cases.
In all cases, C+AG did as about as well as or better than whichever of
AG or CG-Descent performed better.
    
\section{Conclusions}
\label{sec:conc}

We have presented an algorithm called C+AG, which is a variant of
nonlinear conjugate gradient tailored to smooth, convex objective functions.
It has a guaranteed rate of convergence equal to that of accelerated gradient,
which is optimal for the class of problems under consideration in the
function-gradient evaluation model of computation. The method reduces
to linear conjugate gradient in the case that the objective function is
quadratic.

The code and all the test cases are implemented in Matlab and available on
Github under the project named `ConjugatePlusAcceleratedGradient'.

The C+AG code outperformed both AG and CG on most test cases tried.
One interesting future direction is as follows.  The method switches between
CG and AG steps depending on the progress measure.  A more elegant approach
would be a formulation that interpolates continuously between the two
steps.  Indeed, it is possible (see, e.g., \cite{KarimiVavasis-preprint}) to
write down a single parameterized update step in which one parameter
value corresponds to AG and another to CG.  But we do not know of a
systematic way to select this parameter that leads to a complexity bound.

Another direction to pursue is incorporating restarts into AG.  In particular,
\cite{OdonoghueCandes} propose a restart method that appears to work well
in practice according to their experiments.  Both AG and C+AG could use their
restarts, although presumably the former would benefit more than the latter
since AG steps are not used as frequently in C+AG compared to pure AG.
However, to the best of our
knowledge, the restart method is not guaranteed in the sense: It has not
been proved that the bounds on AG iteration of Theorem~\ref{thm:nestthm} are
still valid in the presence of restarts.  Since our goal here was to provide
a method with the iteration guarantee of AG, we did not incorporate
their restarts.

Yet another extension would be to use memory of previous iterates.  In the
context of NCG, adding memory yields L-BFGS \cite{NocedalWright}.  It is
clearly possible to hybridize L-BFGS with AG, and it is likely that the
progress measure proposed still holds for L-BFGS in the case that it is
applied to a quadratic since L-BFGS reduces to LCG in this case.  Note
that some authors, e.g.\ \cite{ZhangODonoghueBoyd},
have also proposed adding memory to AG-like methods, so many combinations
are possible.

Another interesting direction is to develop a conjugate gradient iteration
suitable for constrained convex problems.  Conjugate gradient is known to extend
to the special case of minimizing a convex quadratic function with a Euclidean-ball
constraint;
see e.g. Gould et al.\ \cite{doi:10.1137/S1052623497322735}
and further remarks on the complexity
in Paquette and Vavasis \cite{PaquetteVavasis}.

\section{Acknowledgements}
The authors are grateful to the referees of the previous version of
this paper for their helpful comments.

\appendix

\section{Restarting CG after AG}
\label{sec:restartcg}

As mentioned in Section~\ref{sec:switching}, if the sufficient-progress test fails
for a CG iteration, then the method can fall back to AG iterations.
After the AG termination test holds,
the method resumes CG, say at iteration
$m$.

An issue with restarting CG is that the sufficient-progress
test may not hold for the CG iterations even if the function is purely quadratic
on the remaining iterations.  This is because the proof of sufficient
progress
in Theorem~\ref{thm:quad_suffprog} assumed that $\v_0=\x_0$.  However,
we would not expect the equality
$\x_m=\v_m$ to hold, where $m$ is the first iteration of
CG, and therefore Theorem~\ref{thm:quad_suffprog} does not apply.

We can address this issue with a more elaborate CG sufficient-progress test
after resuming CG.  The more elaborate test increases the number of
function-gradient evaluations per CG iteration from two to four.
This extra overhead, according to our experiments, overwhelmed any
benefit from the more elaborate procedure.

We did observe the benefit of the elaborate procedure in a contrived test
case.  In this contrived test, the objective function $f$ is quadratic, and
the C+AG code is modified in two ways.  First, the initial block of iterations
is AG instead of CG, i.e., the \verb+only_ag+ flag is set to \verb+true+ at
initialization.  Second, instead of setting $\v_0:=\x_0$, we add random noise
to the right-hand side.  For this contrived test, indeed we observed for some
choices of the random noise that CG was never able to restart unless the elaborate
progress measure computation was enabled, and therefore the number of function-gradient
evaluations associated with the elaborate procedure was much smaller than
without it.  However, we did not observe this behavior in any
naturally occurring test case.

Therefore, our default setting is
not to use the elaborate test.
This means that with our default settings,
there is no theoretical guarantee that CG iterations 
can continue indefinitely upon their resumption after AG
even if the objective function is purely quadratic on
the level set of $\x_m$
when CG resumes.  However, the other theoretical guarantees of our method
(that it has the same complexity as accelerated gradient and that it reduces to
CG on a quadratic function) remain.

Despite the fact that the test is turned off by default,
it is available in the code and may prove useful,
and therefore we describe it here for the sake of completeness.
Also, the proof-sketch of $n$-step quadratic convergence
in the next appendix requires that this procedure be used.

Above, we defined $m$ to
be the iteration at which CG iterations are resumed.  To simplify
notation in the remainder of the section, assume $m=0$.  Thus, we consider CG
starting with $\v_0\ne \x_0$.
As in Theorem~\ref{thm:quad_suffprog}, we assume the objective
function $f(\x)$ is quadratic, say $f(\x)=\x^TA\x/2-\b^T\x$.
Recall that linear CG generates a sequence of vectors
$\p_0,\p_1,\ldots$ (see, e.g., \cite{NocedalWright}, \S5.1) such
that $\p_i^TA\p_j=0$ whenever $i\ne j$ (called ``conjugacy'') and
such that for all $k$, $\Span\{\p_0,\ldots,\p_{k-1}\}=
\Span\{\p_0,A\p_0,\ldots,A^{k-1}\p_0\}=\Span\{\nabla f(\x_0),\ldots,\nabla f(\x_{k-1})\}$,
called the $k$th Krylov space and denoted $K_k$ for $k=1,2,\ldots$ (where
$K_0:=\{\bz\}$).  Other properties are that
$\x_k$ is the minimizer of $f$ over the affine
set $\x_0+K_{k}$ and the equation $\x_{k+1}:=\x_k+\alpha_k\p_k$.

The new progress-checking
procedure augments the Krylov space with two additional directions
on each iteration
to assure that the analogs of \eqref{eq:cgok4} and \eqref{eq:cgok6},
that is, \eqref{eq:cgok4a} and \eqref{eq:cgok6a} below,
hold for the method.  In more detail, it
computes additional sequences of vectors
$\w_0,\w_1,\w_2,\ldots$;
$\z_1,\z_2,\z_3,\ldots$;
and $\bar\x_0,\bar\x_1,\bar\x_2\ldots$;
and scalars
$\bar\alpha_0,\bar\alpha_1,\bar\alpha_2,\ldots$;
$\bar\beta_1,\bar\beta_2,\bar\beta_3,\ldots$;
and
$\mu_1,\mu_2,\mu_3,\ldots$ whose formulas are forthcoming.
Let $\bar K_0:=\Span\{\w_0\}$ and for $k\ge 1$,
$\bar K_{k}:=\Span\{\p_0,\ldots,\p_{k-1},\w_k,\z_k\}=K_{k}\oplus\Span\{\w_k,\z_k\}$.
For each $k\ge 1$, define
\begin{equation}
  \u_k := \bar\x_{k-1}-\nabla f(\bar\x_{k-1})/L-\x_k,
  \label{eq:ukdef}
\end{equation}
which is not computed by the method but used in its derivation.
The construction of these vectors and scalars inductively has the following properties.
\begin{enumerate}
\item
  For all $k\ge 0$, $\v_k-\x_k\in\bar K_k$. Furthermore, when writing
  $\v_k-\x_k$ as a linear combination of $\p_0,\ldots,\p_{k-1},\w_k,\z_k$,
  the coefficients of $\w_k,\z_k$ are $1,0$  respectively.
  For this property and the next, in the degenerate case that
  $\p_0,\ldots,\p_{k-1},\w_k,\z_k$ are dependent, this property
  is understood to mean that there exists a way to write $\v_k-\x_k$
  with coefficients $1,0$.
\item
  For all $k\ge 1$,
  $\u_k \in\bar K_k$. Furthermore, when writing
  $\u_k$ as a linear combination of $\p_0,\ldots,\p_{k-1},\w_k,\z_k$,
  the coefficients of $\w_k,\z_k$ are $\mu_k,1$ respectively.
\item
  For $k\ge 1$, the vectors $\p_0,\ldots,\p_{k-1},\w_k,\z_k$ of $\bar K_k$ are conjugate
  with respect to $A$.
\item
  For $k=0$, define $\bar\x_0:=\x_0-\bar\alpha_0\w_0$, and for $k\ge 1$, define
  $\bar\x_k:=\x_k-\bar\alpha_k\w_k-\bar\beta_k\z_k$.  
  Observe that $\bar\x_k-\x_k$ lies in $\bar K_k$.
  Then for all $k\ge 0$, $\bar\x_k$ is the minimizer of $f$ over $\x_k+\bar K_k$.
\item
  For all $k\ge 0  $,
  vector $\v_{k+1}$ is computed (as usual) from $\v_{k}$ according to \eqref{eq:v_k}.
\end{enumerate}

To initiate this sequence of properties, let
\begin{equation}
\w_0:=\v_0-\x_0.  \label{eq:w0} 
\end{equation}
This formula implies Property 1 holds.  Properties 2--3 do not apply when $k=0$.
Assure Property 4 by solving a 1-dimensional minimization problem
for $\bar\alpha_0$, that is,
\begin{equation}
  \bar\alpha_0:=\argmin_{\alpha}f(\x_0-\alpha\w_0)
  \label{eq:a0b0}
\end{equation}
Finally, Property 5
specifies how $\v_1$ is computed.

Next, for the induction step, assume all the above properties
hold for $k$.  Recall that \eqref{eq:v_k} is used to
obtain $\v_{k+1}$.  Observe from \eqref{eq:v_k}
that $\v_{k+1}$ is written
as a linear combination
\begin{equation}
  \v_{k+1}:=\sigma\v_{k}+(1-\sigma)\bar\x_{k} - \tau\nabla f(\bar\x_{k}),
  \label{eq:v_k2}
\end{equation}
where the scalars $\sigma,\tau$ may be read off from \eqref{eq:v_k}.
Note: we omit the subscripts $k$ on $\sigma$, $\tau$ and some
other scalars in this discussion for clarity, since $k$ is fixed.
In the upcoming formulas, if $k=0$, take $\bar\beta_0$, $\p_{-1}$, and $\z_0$ to be zero.
Let us define
\begin{align}
  \w_{k+1}&:=(\sigma-(1-\sigma)\bar\alpha_k)\w_k-
  (1-\sigma)\bar\beta_k\z_k
  +\tau\bar\alpha_kA\w_k
  +\tau\bar\beta_kA\z_k+\delta_1\p_{k-1}+\delta_2\p_k, \label{eq:wk1def} \\
  \z_{k+1}&:=-\bar\alpha_k\w_k-\bar\beta_k\z_k+\bar\alpha_kA\w_k/L+\bar\beta_kA\z_k/L
  +\eps_1\p_{k-1}+\eps_2\p_{k}-\mu_{k+1}\w_{k+1},
  \label{eq:zk1def}
\end{align}
where $\delta_1,\delta_2,\eps_1,\eps_2,\mu_{k+1}$
are yet to be determined.
We now establish
that these formulas satisfy the above properties.

Starting with Property 1, we have
\begin{align}
  \v_{k+1}-\x_{k+1} &=\v_k-\x_k + (1-\sigma)(\bar\x_k-\v_k)
  -\tau\nabla f(\bar\x_k)-\alpha_k\p_k \label{eq:vx1} \\
  &= \sigma(\v_k-\x_k) + (1-\sigma)(\bar\x_k-\x_k)-\tau(A\bar\x_k-\b)-\alpha_k\p_k \label{eq:vx2}\\
  &=\sigma(\v_k-\x_k) + (1-\sigma)(\bar\x_k-\x_k)-\tau(A\x_k-\b)
  -  \tau A(\bar\x_k-\x_k) -
  \alpha_k\p_k. 
  \label{eq:vx3}
\end{align}
Here, \eqref{eq:vx1} follows by substituting \eqref{eq:v_k2} for $\v_{k+1}$
and $\x_{k+1}:=\x_k+\alpha_k\p_k$ for $\x_k$;
\eqref{eq:vx2} follows from rearranging
and from
substituting the definition of $f$;
\eqref{eq:vx3} adds and subtracts $\tau A\x_k$.

Consider the terms of \eqref{eq:vx3} separately.  The first term $\sigma(\v_k-\x_k)$
lies in $\bar K_k$ with coefficients of $\sigma,0$ on $\w_k,\z_k$
respectively by Property 1 inductively.  The second term
$(1-\sigma)(\bar\x_k-\x_k)$
lies in $\bar K_k$ with coefficients $-(1-\sigma)\bar\alpha_k,-(1-\sigma)\bar\beta_k$
on $\w_k,\z_k$ respectively by Property 4 inductively.
The third term $-\tau(A\x_k-\b)$
lies in $\Span\{\p_0,\ldots,\p_k\}$
by the usual properties of CG.  The fourth term 
$-\tau A(\bar\x_k-\x_k)$ lies in $A\Span\{\p_0,\ldots,\p_{k-1},\w_k,\z_k\}$=
$\Span\{\p_1,\ldots,\p_k,A\w_k,A\z_k\}$.  In this expansion, the coefficients
of $A\w_k$ and $A\z_k$ are $\tau\bar\alpha_k,\tau\bar\beta_k$ respectively.
Finally, the last term obviously lies in
$\Span\{\p_k\}$.
Therefore, if we
define $\w_{k+1}$ according to \eqref{eq:wk1def}, then Property 1 holds
for $k+1$.

Next we turn to Property 2.  Observe that
\begin{align}
  \u_{k+1}&=\bar\x_k-\nabla f(\bar\x_k)/L -\x_{k+1} \label{eq:uk1a} \\
  &=(\bar\x_k-\x_k)-\alpha_k\p_k-(A\x_k-\b)/L - A(\bar\x_k-\x_k)/L. \label{eq:uk1b}
\end{align}
Here, \eqref{eq:uk1a} follows from the definition \eqref{eq:ukdef}
while \eqref{eq:uk1b} follows from the substitutions
$\nabla f(\bar\x_k)=A\bar\x_k-\b$ and $\x_{k+1}=\x_k+\alpha_k\p_k$ and
adding and subtracting $A\x_k/L$.

Considering \eqref{eq:uk1b} term by term,
the first term $(\bar\x_k-\x_k)$ lies in $\bar K_k$ with coefficients
$-\bar\alpha_k,-\bar\beta_k$ on $\w_k,\z_k$ inductively by Property 4.
The second term lies in $\bar K_{k+1}$.
The third term $-(A\x_k-\b)/L$ lies in $\Span\{\p_0,\ldots,\p_k\}$.
As in the preceding derivation, the fourth term $- A(\bar\x_k-\x_k)/L$
lies in $\Span\{\p_0,\ldots,\p_k,A\w_k,A\z_k\}$ with coefficients
$\bar\alpha_k/L,\bar\beta_k/L$ on $A\w_k,A\z_k$ respectively.
Again, we see now that \eqref{eq:zk1def} is chosen precisely so
that the induction carries through on Property 2.

Now we turn to Property 3.
By the induction hypothesis on Property 3,
$\w_k,\z_k$ are conjugate to each other and to $\p_0,\ldots,\p_{k-1}$.  This implies
$A\w_k,A\z_k$ are conjugate to $\p_0\ldots,\p_{k-2}$ since for $j=0,\ldots,k-2$,
$\p_j^TA(A\w_k)=(A\p_j)^TA\w_k$, and $A\p_j\in \Span\{\p_0,\ldots,\p_{k-1}\}$, and
similarly for $A\z_k$.
Thus, we see term-by-term that every term in the definitions of $\w_{k+1},\z_{k+1}$
in \eqref{eq:wk1def}--\eqref{eq:zk1def}
is already conjugate to $\p_0,\ldots,\p_{k-2}$.  Thus,
to ensure $\w_{k+1}$ is also conjugate to $\p_{k-1}$ and $\p_{k}$, solve
a system of two linear equations for the scalars $\delta_1,\delta_2$
that appear in \eqref{eq:wk1def}.  (In fact, by conjugacy,
the system decouples into two univariate linear equations that are solved
via single division operations.)  Similarly, to assure $\z_{k+1}$ is
conjugate to $\p_{k-1},\p_k,\w_{k+1}$, we solve three linear equations for
$\eps_1,\eps_2,\mu_{k+1}$ from
\eqref{eq:zk1def}, which again decouple into three separate equations.

Next, for Property 4, recall again from the theory of
CG that minimizing $f(\x)$ over $\x_{k+1}+\bar K_{k+1}$  for which we have a conjugate
basis of $\bar K_k$ reduces to $k+1$ uncoupled 1-dimensional minimizations; indeed,
this is the main purpose of conjugacy.  However, $\x_{k+1}$ is already minimal with
respect to $\p_0,\ldots,\p_k$ by the usual properties of CG, so it
is necessary to minimize only with respect to $\w_{k+1},\z_{k+1}$.  Therefore, to
obtain $\bar\x_{k+1}$, one defines
\begin{equation}
  \bar\x_{k+1}:=\x_{k+1}-\bar\alpha_{k+1}\w_{k+1}-\bar\beta_{k+1}\z_{k+1},
  \label{eq:xk1}
\end{equation}
where
\begin{equation}
  (\bar\alpha_{k+1},\bar\beta_{k+1}):=\argmin_{\alpha,\beta} f(\x_{k+1}-\alpha\w_{k+1}-\beta\z_{k+1}),
  \label{eq:alphabetak1}
\end{equation}
a problem that decouples into two simple equations.

Thus, we have established the following result:
\begin{theorem}
  Assume the sequences $\bar\alpha_k$'s,
  $\bar\beta_k$'s, $\mu_k$'s,
  $\v_k$'s, $\w_k$'s,
  $\z_k$'s, and $\bar\x_k$'s are defined by \eqref{eq:v_k} and
  \eqref{eq:w0}--\eqref{eq:alphabetak1}.  Then Properties $1$--$5$ above hold.
\end{theorem}

Now, finally, we can propose a progress measure assured to hold for quadratic
functions, namely, $f(\bar\x_k)\le \phi^*_k$.  The proof that this holds
is a straightforward rewriting of
Theorem~\ref{thm:quad_suffprog}.

\begin{theorem}
  If $f(\x)=\frac{1}{2}\x^TA\x-\b^T\x$, where $A$ is positive definite,
  then $f(\bar\x_k)\le \phi_k^*$ for every iteration
  $k$, where $\x_k$ is the sequence of conjugate gradient iterates
  and $\bar\x_k$ is defined by \eqref{eq:xk1}.
\end{theorem}

\begin{proof}
  The result clearly holds when $k=0$ by \eqref{eq:phi0} and the fact
  that $f(\bar\x_0)\le f(\x_0)$ by the minimality of $\bar\x_0$ in
  the line $\x_0+\Span\{\w_0\}$.
  Assuming $f(\bar\x_k)\le \phi_k^*$, we now show 
  that $f(\bar\x_{k+1})\le \phi_{k+1}^*$.
  For $k\ge 0$,
  \begin{align}
    \phi_{k+1}^*
    &=
    (1-\theta_k)\phi_k^*+\theta_k f(\bar\x_k)
    -\frac{\theta_k^2}{2\gamma_{k+1}}\Vert \nabla f(\bar\x_k)\Vert^2 \notag\\
    &\hphantom{}\quad\mbox{}
    +\frac{\theta_k(1-\theta_k)\gamma_k}{\gamma_{k+1}}
    \left(\ell \Vert \bar\x_k-\v_k\Vert^2+
    \nabla f(\bar \x_k)^T(\v_k-\bar\x_k)\right)\label{eq:cgok1a} \\
    &\ge
    f(\bar\x_k)
    -\frac{\theta_k^2}{2\gamma_{k+1}}\Vert \nabla f(\bar \x_k)\Vert^2
    +\frac{\theta_k(1-\theta_k)\gamma_k}{\gamma_{k+1}}
    \left(\nabla f(\bar\x_k)^T(\v_k-\bar\x_k)\right) \label{eq:cgok2a}\\
    & = 
    f(\bar\x_k)
    -\frac{1}{2L}\Vert \nabla f(\bar\x_k)\Vert^2
    +\frac{\theta_k(1-\theta_k)\gamma_k}{\gamma_{k+1}}
    \left(\nabla f(\bar\x_k)^T(\v_k-\bar\x_k)\right)\label{eq:cgok3a} \\
    & =
    f(\bar\x_k)
    -\frac{1}{2L}\Vert \nabla f(\bar\x_k)\Vert^2\label{eq:cgok4a} \\
    & \ge 
    f(\bar\x_k - \nabla f(\bar\x_k)/L)
    \label{eq:cgok5a} \\
    & \ge
    f(\bar\x_{k+1}).\label{eq:cgok6a}
  \end{align}
  Here, \eqref{eq:cgok1a} is a restatement of \eqref{eq:phi_k}.
  Line \eqref{eq:cgok2a} follows
  from the induction hypothesis.  Line \eqref{eq:cgok3a}
  follows from \eqref{eq:thetaLrel}.

  Line \eqref{eq:cgok4a} follows because,
  according to the recursive formula
  \eqref{eq:v_k},
  $\v_k$ lies in the affine space
  $\x_k+\bar K_k$.
  Also, $\bar\x_k$ lies in this affine space and is optimal
  for $f$ over this space.  Therefore, $\nabla f(\bar\x_k)$
  is orthogonal to every vector in $\bar K_k$ and in particular to $\v_k-\bar\x_k$.

  Line \eqref{eq:cgok5a} is the
  standard inequality for $L$-smooth
  convex functions.
  Finally \eqref{eq:cgok6a}
  follows
  because $\x_k-\nabla f(\x_k)/L=\x_{k+1}+\u_{k+1}$ by \eqref{eq:ukdef}
  and $\u_{k+1}\in \bar K_{k+1}$ by Property 2.  On the other hand
  $\bar \x_{k+1}$ minimizes $f$ over $\x_{k+1}+\bar K_{k+1}$.
\end{proof}

This concludes the discussion of how to compute $\bar\x_{k+1}$ in the case
that $f$ is quadratic.  In the general case, we do not have a matrix $A$
but only $f$ and $\nabla f$.  Therefore, we propose a method to evaluate
the quantities in the preceding paragraphs in the general case.
Assume inductively that we have $\w_k,\z_k,\hat\w_k,\hat\z_k$,
where $\hat\w_k$ corresponds to $A\w_k$ in the above equations
and similarly $\hat\z_k$ corresponds to $A\z_k$.
Then all the formulas above may be evaluated using these vectors.
Note that we already have vectors corresponding to $A\p_{k-1}$
and $A\p_k$; these are found in \eqref{eq:alphak1}--\eqref{eq:alphak3}.
We need these vectors to solve for the scalars $\delta_1,\delta_2,\eps_1,\eps_2$
assuring conjugacy.

To continue the induction, it is necessary to
evaluate $\hat\w_{k+1},\hat\z_{k+1}$ corresponding to
$A\w_{k+1},A\z_{k+1}$.  Multiplying both sides of \eqref{eq:wk1def} and \eqref{eq:zk1def}
by the fictitious $A$ shows that we need to obtain the analog of the product
$A(\bar\alpha_k A\w_{k} + \bar\beta_{k} A\z_{k})$. Note that this combination
of $A\w_{k},A\z_{k}$ appears in both right-hand sides of
\eqref{eq:wk1def} and \eqref{eq:zk1def}.
This may be found by evaluating
$\nabla f(\x_{k+1}+\bar\alpha_k \hat\w_k+\bar\beta_k\hat\z_k)$
and then subtracting $\nabla f(\x_{k+1})$ (which is
already evaluated).

We also need to evaluate
$f(\bar\x_{k+1})$ and $\nabla f(\bar\x_{k+1})$
in order to obtain \eqref{eq:ukdef} and check the progress measure
$f(\bar\x_{k+1})\le \phi_{k+1}^*$.  
Thus, the cost of carrying
out the method of this section is two extra function-gradient evaluations per
iteration.  As mentioned earlier, our computational experiments indicate that
this 100\% overhead is not compensated for by any reduction in iteration count
in any test case,
and therefore the method described in this section is disabled by default.

\section{Toward $n$-step quadratic convergence}
\label{sec:nstep}

In previous sections, we showed the proposed C+AG method has the
same complexity bound up to a constant factor as accelerated gradient.
In this section we suggest that it has an additional complexity bound
known to apply to nonlinear conjugate gradient, namely, $n$-step
quadratic convergence.  This property is mainly of theoretical interest
since $n$-step quadratic convergence is not usually observed in
practice (see, e.g., \cite[p.~124]{NocedalWright}) for larger problems.  Therefore,
we will provide only a sketch of a possible proof and leave the
proof, which is likely to be complicated, to future research.

We propose to
follow the proof given by McCormick and Ritter \cite{McCormickRitter}
rather than the original proof by Cohen \cite{Cohen}.  They argue that
classical NCG is $n$-step quadratically convergent provided that
the line-search is sufficiently accurate and provided that
the objective function is strongly convex in the neighborhood of the
objective function and that its second derivative is Lipschitz continuous.
Finally, they assume that nonlinear CG is restarted every $n$ iterations.

In order to apply this result to C+AG, we first need to argue that the
sufficient-progress test will not prevent CG iterations.  We showed
in the last section that for quadratic functions, the sufficient-progress
test always holds.
A strongly convex function with a Lipschitz second derivative is
approximated arbitrarily closely by a quadratic function in a sufficiently
small neighborhood of the root.  In particular, if we assume that
$L$ is strictly larger than the largest eigenvalue of the Hessian in
a neighborhood of the root, then inequality
\eqref{eq:cgok5a} is strictly satisfied, and small deviations from
a quadratic function are dominated by the residual
in \eqref{eq:cgok5a}.

Some other issues with \cite{McCormickRitter} is that their version of NCG implements
a step-limiter operation that is not present in our algorithm.  It is not
clear whether the step-limiter is a requirement for $n$-step quadratic convergence
or merely a device to simplify their analysis.  Furthermore, \cite{McCormickRitter}
assume the Polak-Ribi\`ere choice of $\beta_k$
\cite{NocedalWright} rather than
the Hager-Zhang step that we used.
Note that \cite{McCormickRitter} predates \cite{HagerZhang} by
several decades.  We conjecture that the analysis of \cite{McCormickRitter} extends to
\cite{HagerZhang}.  

The line-search accuracy required by \cite{McCormickRitter} is as follows.
On each iteration $k$,
$|\alpha_{k}-\alpha_{k}^*|\le O(\Vert\p_{k}\Vert)$, where
$\alpha_{k}$ is chosen by C+AG in \eqref{eq:alphak1}--\eqref{eq:alphak3} and
$\alpha_{k}^*=\argmin\{f(\x_k+\alpha\p_{k}):\alpha\in\R\}$.
We sketch an argument that could prove
$|\alpha_{k}-\alpha_{k}^*|\le O(\Vert\p_{k}\Vert)$
as follows.  

Let $\psi(\alpha):=f(\x_k+\alpha\p_{k})$.  It follows from \cite{McCormickRitter} that
$\Vert\p_{k}\Vert=\Theta(\Vert\x_k-\x^*\Vert)$.  In other words, there is
both an upper and lower bound on
$$\frac{\Vert\p_{k}\Vert}
{\Vert\x_k-\x^*\Vert}$$
independent of $k$.
It also follows from their results that $\psi'(0)<0$.
Write $\psi''(\alpha)=s+\epsilon(\alpha)$, where $s>0$ is the
estimate of $\psi''$ obtained from \eqref{eq:alphak1}--\eqref{eq:alphak3},
that is $s = L(\psi'(1/L)-\psi'(0))$.
It follows by strong convexity and the equation
$\Vert\p_{k}\Vert=\Theta(\Vert\x_k-\x^*\Vert)$
that $\alpha_{k}\le O(1)$
and $\alpha_{k}^*\le O(1)$.
Again by these same assumptions, we can also derive
that $s=\Theta(\Vert\p_{k}\Vert^2)$. 
By the Lipschitz assumption, $\epsilon(\alpha)\le O(\Vert\p_{k}\Vert^3)$
for $\alpha\le O(1)$.

It also follows from the assumption of strict convexity and proximity to the root
$|\psi'(0)|\le O(\Vert\p_{k}\Vert^2)$.  The true minimizer $\alpha_{k}^*$ is the
root of $\psi'$ and therefore satisfies the equation
$$\psi'(0)+\int_0^{\alpha_{k}^*} (s+\eps(\alpha))\,d\alpha = 0,$$
i.e.,
$$\psi'(0)+s\alpha_{k}^* + \int_0^{\alpha_{k}^*} \eps(\alpha)\,d\alpha = 0.$$
The computed minimizer is $\alpha_{k}=-\psi'(0)/s$.
The integral in the previous line is bounded by $O(\Vert\p_{k}\Vert^3)$,
whereas $-\psi'(0)$ is $O(\Vert\p_{k}\Vert^2)$, and $s=\Theta(\Vert\p_{k}\Vert^2)$.
Therefore, the intermediate value theorem applied to the above
univariate equation for $\alpha_{k}^*$ tells us that
$\alpha_{k}^*=-\psi'(0)/s+O(\Vert\p_{k}\Vert)=\alpha_k+O(\Vert\p_{k}\Vert)$.
This concludes the explanation of accuracy of the line-search.

\bibliographystyle{plain}
\bibliography{optimization.bib}

\end{document}